\documentclass[a4paper, 10pt, twoside]{article}

\usepackage{amsmath, amscd, amsfonts, amssymb, amsthm, latexsym, url, color, todonotes, bm, framed, rotating, enumerate} 
\usepackage{graphicx}
\usepackage[left=1in, right=1in, top=1.2in, bottom=1in, includefoot, headheight=13.6pt]{geometry}
\usepackage{mathtools}
\input{xy}
\xyoption{all}
\usepackage{caption} %

\usepackage{tikz}
\usetikzlibrary{babel}
\usepackage{quiver}
\usepackage{thm-restate}

\usepackage[colorlinks,breaklinks=true]{hyperref}
\usepackage[figure,table]{hypcap}
\hypersetup{
	bookmarksnumbered,
	pdfstartview={FitH},
	citecolor={black},
	linkcolor={black},
	urlcolor={black},
	pdfpagemode={UseOutlines}
}
\makeatletter
\newcommand\org@hypertarget{}
\let\org@hypertarget\hypertarget
\renewcommand\hypertarget[2]{%
  \Hy@raisedlink{\org@hypertarget{#1}{}}#2%
} 
\makeatother 

\newtheorem{theorem}{Theorem}[section]
\newtheorem{lemma}[theorem]{Lemma}
\newtheorem{corollary}[theorem]{Corollary}

\theoremstyle{definition}
\newtheorem{definition}[theorem]{Definition}
\newtheorem{remark}[theorem]{Remark}

\newcommand{\xysquare}[8]{
\[\xymatrix{
#1 \ar@{#5}[r] \ar@{#6}[d] & #2 \ar@{#7}[d]\\
#3 \ar@{#8}[r] & #4
}\]
}

\newcommand{\bb}{\mathbb}

\newcommand{\comment}[1]{}

\renewcommand{\phi}{\varphi}

\newcommand{\xto}{\xrightarrow}

\renewcommand{\cal}{\mathcal}

\newcommand{\indlim}{\varinjlim}

\renewcommand{\Im}{\operatorname{Im}}

\DeclareMathOperator{\Frac}{Frac}

\DeclareMathOperator{\Spec}{Spec}
\DeclareMathOperator{\Proj}{Proj}

\DeclareMathOperator{\Tor}{Tor}

\DeclareMathOperator{\essdim}{essdim}

\newcommand{\CH}{C\!H}

\newcommand{\rcoeq}[3]{\xymatrix{ #1\ar@/^3mm/[r]^f \ar@/_3mm/[r]_g & #2 \ar[l]_e\ar[r] & #3}}

\usepackage[bbgreekl]{mathbbol}
\DeclareSymbolFontAlphabet{\mathbbm}{bbold}

\usepackage{fancyhdr}

\usepackage{sectsty}
\sectionfont{\Large\sc\centering}
\chapterfont{\large\sc\centering}
\chaptertitlefont{\LARGE\centering}
\partfont{\centering}

\begin{document}
\itemsep0pt

\title{On the decomposition of the diagonal for n-fold hypersurfaces via cubic n-1-folds}

\author{Elia Fiammengo and Morten Lüders}

\date{}

\maketitle

\begin{abstract}
We study the question of the existence of a decomposition of the diagonal for very general degree $d\geq 4$ and $(2,d\geq 3)$-complete intersection $n$-folds. Using cycle-theoretic techniques of Lange, Pavic and Schreieder we reduce the question to the existence of such a decomposition for $n-1$-dimensional cubic hypersurfaces. This extends results of Nicaise and Ottem from stable rationality to the existence of a decomposition of the diagonal. A key ingredient is the essential dimension of cubic hypersurfaces due to Voisin. As an application, we use a recent result of Engel, de Gaay Fortman and Schreieder on the decomposition of the diagonal for cubic threefolds to give a new proof of the retract irrationality of a very general complex quartic $4$-fold, originally due to Totaro, and of a very general complex $(2,3)$-complete intersection $4$-fold, originally due to Skauli. As another application we show that a very general quintic fourfold containing a three-dimensional linear subspace with multiplicity $2$, which is unirational due to results of Massarenti, does not admit a decomposition of the diagonal and is therefore not retract rational.
\end{abstract}

\begingroup

\smallskip

\par
\endgroup

\medskip


\section{Introduction}
The goal of birational geometry is to classify algebraic varieties up to birational isomorphism.
Evidently, one of the first important questions in the field is to determine if or which varieties are birational to projective space $\bb P^n$. Such varieties are called rational. A variety $X$ 
is called stably rational if $X\times \bb P^m$ is birational to $\bb P^n$ for some $m,n\geq 0$.  A variety $X$ is called retract rational if for some $n\geq \dim X$ there exist rational maps $f:X\dashrightarrow \bb P^n$ and $g:\bb P^n\dashrightarrow X$ such that the composition $g\circ f$ is the identity. There are the following implications between these notions for proper smooth varieties:
$$\text{rational} \Rightarrow \text{stably rational} \Rightarrow\text{retract rational} \Rightarrow \text{rationally chain connected}. $$

In 2014 Voisin proposed a new idea to disprove stable rationality (and even retract rationality and $\bb A^1$-connectivity) realizing that such varieties admit a decomposition of the diagonal and that this property specializes in families. Therefore degenerating a variety to a singular variety which does not admit a decomposition of the diagonal, one can disprove stable rationality. This technique has since been developed further by Colliot-Thél\`ene--Pirutka \cite{CTP2016}, Totaro \cite{Totaro2016} and Schreieder \cite{ SchreiederAMS} and has had many applications.

Using a completely different method, i.e. the existence of a version of the nearby cycles functor for stable birational types developed by Nicaise--Shinder \cite{NicaiseShinder2019} and Kontsevich--Tschinkel \cite{KontsevichTschinkel},  Nicaise--Ottem \cite{NicaiseOttem2022} showed in 2022 that obstructions to stable rationality may lie not just in the irreducible components of the degeneration but also in its lower dimensional strata. This was extensively used in \cite{NicaiseOttem2022} and \cite{Moe2023} to find new examples of stably irrational varieties.

In \cite{PavicSchreieder2023} and \cite{LangeSchreieder2024}, Pavic--Schreieder and Lange--Schreieder developed a cycle theoretic analogue of the method of Nicaise--Ottem which is based on the decomposition of the diagonal and hence allows one to find obstructions to retract rationality and $\bb A^1$-connectivity in lower dimensional strata. Using their cycle theoretic obstructions, many of the results of \cite{NicaiseOttem2022} and \cite{Moe2023} have been upgraded (see e.g. \cite{LangeSkauli2023}, \cite{LangeZhang2025}).

In this article, we use the method developed in \cite{LangeSchreieder2024}, together with the closely related notion of essential dimension (see Definition \ref{def_essdim}), to show that one can reduce the non-existence of a decomposition of the diagonal for hypersurfaces and of certain complete intersections to that of lower dimensional cubics.
Our main result is inspired by  \cite[Cor. 7.9, Thm. 4.4]{NicaiseOttem2022} and reads as follows:
\begin{restatable}{theorem}{Maintheorem}\label{Theorem_2_4_complete_int_Intro}
Let $k=\bb C$. Let $n\geq 3$ and $d\geq 4$ be positive integers. Assume that a very general cubic hypersurface in $\bb P^{n+1}_k$ does not admit a decomposition of the diagonal. Then a very general member of the following families 
\begin{enumerate}
 \item hypersurfaces of degree $d$ in $\bb P^{n+2}_k$,
    \item complete intersections of type $(2,d-1)$ in $\bb P^{n+3}_k$,
   
    \end{enumerate}
    does not admit a decomposition of the diagonal and is therefore neither retract rational nor $\bb A^1$-connected.
\end{restatable}

Using a recent result by Engel--de Gaay Fortman--Schreieder showing that a very general cubic threefold does not admit a decomposition of the diagonal \cite{EFS2025}, we obtain the following corollary:
\begin{restatable}{corollary}{Corollary_main_Intro}\label{Corollary_main_Intro}
Let $k=\bb C$ and $d\geq 4$, then a very general member of the following families 
\begin{enumerate}
     \item hypersurfaces of degree $d$ in $\bb P^{5}_k$,
      \item complete intersection of type $(2,d-1)$ in $\bb P^{6}_k$,
\end{enumerate}
does not admit a decomposition of the diagonal and is therefore neither retract rational nor $\bb A^1$-connected.
\end{restatable}
\begin{remark}
For $d=4$ Corollary \ref{Corollary_main_Intro}(i) is originally due to Totaro \cite{Totaro2016} and Corollary \ref{Corollary_main_Intro}(ii) is originally due to Skauli \cite{Skauli2023}. For more results on the stable and retract rationality of complete intersections we refer to \cite{ChatzistamatiouLevine2018}, \cite[Sec. 7]{NicaiseOttem2022} and \cite{LangeZhang2025}.    
\end{remark}

\begin{remark}\label{remark_2_divisibility}
    Our proof of Theorem \ref{Theorem_2_4_complete_int_Intro} (see Footnote 1) shows that, under the given assumptions, the torsion order of a very general member $X$ of families $(i)$ and $(ii)$ is divisible by $\Tor(C)$, where $C\subset \bb P^{n+1}$ is a very general cubic. As $\Tor(C)=2$ by \cite[Cor. 7.12]{Schreiedersurvey}, it follows that $\Tor(X)$ is even. This is known to hold unconditionally for a very general hypersurface $X\subset \bb P^{n+r+s+1}$ of degree $d\geq 2+n$ if $r\leq 2^n-2$ and $s\leq (\lfloor\frac{n}{m}\rfloor-1)(2^{n-1}-1)$ by \cite[Thm. 1.3]{LangeSchreieder2024}. 
\end{remark}

As a byproduct of the proof of Corollary \ref{Corollary_main_Intro}, and using recent results on unirationality of hypersurfaces due to Massarenti \cite{Massarenti2023}, we find the following new examples of unirational varieties that are not retract rational.
\begin{restatable}{corollary}{cor:unir_retract_irr}\label{cor:unir_retract_irr}
    A very general member of the following singular families
    \begin{enumerate}
        \item quartic fourfolds with a double point,
        \item quintic fourfolds containing a three-dimensional linear subspace with multiplicity $2$,
    \end{enumerate}
     is unirational and does not admit a decomposition of the diagonal and is therefore neither retract rational nor $\bb A^1$-connected.
\end{restatable}

General quartic hypersurfaces of dimension $\geq 5$ are known to be unirational \cite{ConteMurre98}. In dimension four, there are singular examples like the quadric surface bundle of Hassett-Pirutka--Tschinkel \cite{HassetPirutkaTschinkel2018} which are unirational and stably irrational. Corollary \ref{cor:unir_retract_irr}(i) gives new examples with mild singularities.
To the best of our knowledge, (ii) is the first, even though singular,

example of a unirational but not retract rational quintic hypersurface in dimension four.  By results of Conte--Marchisio--Murre \cite{ConteMurre2008} and Schreieder \cite[Cor. 1.4]{SchreiederAMS} it has been known that a general quintic hypersurface $X$ in $\bb P^{n+1}$ containing a three-dimensional linear subspace is unirational and $X$ is not retract rational for $ n\in \{6,7,8,9\}$. In dimension four it is known by results of Okada \cite[Thm. 1]{Okadasubspaces2020} that general quintic containing a three-dimensional linear subspace with multiplicity $1$ is retract irrational but it is not known if it is unirational.

\subsection{Outline of the proofs
}
For the proof of Theorem \ref{Theorem_2_4_complete_int_Intro} we use the cycle-theoretic obstructions recently developed by Lange--Schreieder \cite{LangeSchreieder2024}. These are slightly more flexible than the ones developed by Pavic--Schreieder \cite{PavicSchreieder2023} and use the notion of a relative decomposition of the diagonal (see the next section). The latter notion is closely related to the concept of the essential dimension of a scheme.  The proof of Theorem \ref{Theorem_2_4_complete_int_Intro} closely follows the ideas of \cite[Sec. 5]{LangeSchreieder2024} with the notion of essential dimension as a new input.

The strategy of our argument can be summarized as follows: firstly, by a degeneration argument, the nonexistence of a decomposition of the diagonal of a hypersurface of degree $3$ and dimension $n$ implies the nonexistence of a decomposition of the diagonal for a $(2,3)$-complete intersection of dimension $n+1$. Secondly, this complete intersection can be shown to be birational to a degree $4$ hypersurface of dimension $n+1$:
\begin{equation}\label{equation_argument}
\deg H=3,\;\dim H= n \rightsquigarrow (2,3)\text{-CI},\; \dim n+1 \rightsquigarrow \deg H=4,\;\dim H= n+1.    
\end{equation}
This goes under the slogan ``raising the degree and the dimension'' and we should mention that just ``raising the degree'' is also possible using simpler degenerations (see \cite[Lem. 2.4]{Totaro2016}). The difference from the argument in Lange--Schreieder is firstly that in loc. cit. they start with a hypersurface of degree $d\geq 4$ and take as input Schreieder's singular hypersurfaces and a calculation of their torsion order with respect to specific subschemes. Secondly, in the second step  of the argument they show that their complete intersection is birational to a degree $d$ hypersurface of dimension $n+1$ which means ``raising the dimension without raising the degree''. Thirdly, they calculate the torsion order with respect to the above mentioned closed subschemes needed in their constructions, while we appeal to calculations of the essential dimension of cubic hypersurfaces.

\paragraph{Acknowledgements.} We are grateful to Jan Lange and Stefan Schreieder for answering questions we had about their work and for very helpful comments and suggestions. The second author is funded by the Deutsche Forschungsgemeinschaft (DFG, German Research Foundation), project number $557768455$ and by the Deutsche Forschungsgemeinschaft (DFG, German Research Foundation)
TRR 326 \textit{Geometry and Arithmetic of Uniformized Structures}, project number 444845124.

\section{A cycle-theoretic obstruction to the existence of a relative decomposition of the diagonal}
In this section we review a method due to Lange--Schreieder \cite{LangeSchreieder2024} which can be used to show that a variety does not admit a decomposition of the diagonal relative to a closed subset. The obstruction is formulated in terms of an obstruction map, a variant of which was first developed by Pavic--Schreieder \cite{PavicSchreieder2023}. We also cite facts about the relative torsion order which we will need in the next sections from \cite{LangeSchreieder2024}. 

\begin{definition}\label{Def_relative_decomposition}
    Let $X$ be a variety over a field $k$ and let $\Lambda$ be a ring. Let $\delta_X$ denote the image of the diagonal $\Delta\subset X\times X$ in $\CH_0(X_{k(X)},\Lambda)=\indlim_{V\subset X}\CH_{\dim X}(X\times V,\Lambda)$. Here $\CH_0(X_{k(X)},\Lambda):=\CH_0(X_{k(X)})\otimes_{\bb Z}\Lambda$. We say that $X$ \textit{admits a $\Lambda$-decomposition of the diagonal relative to a closed subset $W\subset X$} if
    $$\delta_X\in \Im (\CH_0(W_{k(X)},\Lambda)\to \CH_0(X_{k(X)},\Lambda)).$$
\end{definition} 
Let the notation be as in the definition and let $U=X-W$. Then by the localisation exact sequence
$$\CH_0(W_{k(X)},\Lambda)\to \CH_0(X_{k(X)},\Lambda)\xto{j^*} \CH_0(U_{k(X)},\Lambda)\to 0$$
the condition that $X$ admits a $\Lambda$-decomposition of the diagonal relative to $W$ is equivalent to $\delta_X$ being in the kernel of $j^*$.
\begin{remark}
 By the localisation sequence there exists some zero-dimensional closed subset $W\subset X$ relative to which it admits a decomposition of the diagonal iff $X$ admits a \textit{decomposition of the diagonal} in the following sense: there exists a zero-cycle $z\in \CH_0(X)$ and a cycle $Z\in Z_{\dim X}(X\times_k X)$ whose support does not dominate the second factor of $X\times_k X$ such that
 $$e\cdot\Delta_X=z\times X+Z\in \CH_{\dim X}(X\times_k X)$$
 for $e=1$. The smallest integer $e$ such that a decomposition as the above exists is also called the \textit{torsion order} of $X$ and denoted by $\Tor(X)$.
\end{remark}

\begin{definition}
    The $\Lambda$-\textit{torsion order} of $X$ relative to a closed subset $W\subset X$, denoted by $\Tor^\Lambda(X,W)$, is the order of the element 
    $$\delta_X|_U=\delta_U\in \CH_0(U_{k(X)},\Lambda).$$
\end{definition}

\begin{remark}\label{Remark_LS_3.4}
    Note that $\Tor^\Lambda(X,W)= \Tor^\Lambda(U,\emptyset).$
Furthermore, $\Tor^\Lambda(X,W)=1$ iff $X$ admits a $\Lambda$-decomposition of the diagonal relative to $W$ \cite[Rem. 3.4]{LangeSchreieder2024}.
\end{remark}

The torsion order has the following important properties which we will need:
\begin{lemma}\label{Lemma_LS_3.6} \cite[Lem. 3.6]{LangeSchreieder2024}
    Let $X$ be a variety over a field $k$ and let $W\subset X$ be closed. Then the following hold:
    \begin{enumerate}
        \item[(a)] For all $m\in \bb Z$, $\Tor^{\bb Z/m}(X,W)|\Tor^{\bb Z}(X,W)$.
        \item[(b)] Let $W'\subset W\subset X$ be closed subsets, then $\Tor^{\Lambda}(X,W)|\Tor^{\Lambda}(X,W')$. 
        \item[(c)] $\Tor(X)$ is the minimum of the relative torsion orders $\Tor^{\bb Z}(X,W)$ where $W\subset X$ runs through all closed subsets of dimension zero.
         \item[(d)] If $\deg :\CH_0(X)\to \bb Z$ is an isomorphism, then $\Tor(X)=\Tor^\bb Z(X,W)$ for any closed subset $W\subset X$ of dimension zero which contains a zero-cycle of degree $1$.
        \item[(e)] If $k=\bar k$ is algebraically closed, then $\Tor^{\Lambda}(X,W)=\Tor^{\Lambda}(X_L,W_L)$ for any ring $\Lambda$ and any field extension $L/k$.
    \end{enumerate}
\end{lemma}


Let $Y=\bigcup_{i\in I}Y_i$ be an snc scheme over $k$ which has no triple intersections and fix a total order on $I$ denoted by $<$. Let $\Lambda$ be a ring and $Y_{ij}=Y_i\cap Y_j$ for $i,j\in I$. Then in \cite[Def. 4.1]{LangeSchreieder2024} the obstruction map
$$\Psi_Y^\Lambda:\bigoplus_{l\in I}\CH_1(Y_l,\Lambda)\to \bigoplus_{\stackrel{i,j\in I}{i<j}}\CH_0(Y_{ij},\Lambda) $$
$$(\gamma_l)_l\mapsto(\gamma_i|_{Y_{ij}}-\gamma_j|_{Y_{ij}})_{i,j}$$
is defined.

\begin{theorem}\label{Theorem_LS_4.1} \cite[Thm. 4.2]{LangeSchreieder2024}
    Let $R$ be a discrete valuation ring with fraction field $K$ and algebraically closed residue field $k$ and $\Lambda$ be a ring of positive characteristic $c\in \bb Z_{\geq 1}$ such that the exponential characteristic of $k$ is invertible in $\Lambda$. Let $\cal X\to \Spec R$ be a strictly semistable $R$-scheme with geometrically integral generic fiber $X$ and special fiber $Y$. Assume that $Y=\bigcup_{i\in I}Y_i$ has no triple intersections and fix a total order $<$ on $I$. Then 
    $$\mathrm{coker} (\Psi_{Y_L}^\Lambda:\bigoplus_{l\in I}\CH_1(Y_l\times_k L,\Lambda)\to \bigoplus_{\stackrel{i,j\in I}{i<j}}\CH_0(Y_{ij}\times_k L,\Lambda)) $$
    is $\Tor^\Lambda(\bar{X},\emptyset)$-torsion for every field extension $L/k$.
\end{theorem}

Choosing a relative Nagata compactification of $\cal X\to \Spec R$, we may assume that $\cal X\to \Spec R$ is a proper flat separated $R$-scheme with geometrically integral generic fiber $X$ and special fiber $Y$. Now if $W_\cal X\subset \cal X$ is a closed subscheme such that
\begin{enumerate}
    \item $\cal X^\circ:= \cal X\setminus W_\cal X$ is a strictly semistable $R$-scheme, and
    \item $Y^\circ:=Y\setminus W_Y=\bigcup_{i\in I}Y_i^\circ$ has no triple intersections,
\end{enumerate}
then the above theorem applies to $\cal X^\circ$ and says that $$\mathrm{coker} (\Psi_{Y_L^\circ}^\Lambda:\bigoplus_{l\in I}\CH_1(Y_{l}^\circ\times_k L,\Lambda)\to \bigoplus_{\stackrel{i,j\in I}{i<j}}\CH_0(Y_{ij}^\circ\times_k L,\Lambda)) $$
    is $\Tor^\Lambda(\bar{X},W_{\bar{X}})=\Tor^\Lambda(\bar{X}^\circ,\emptyset)$-torsion for every field extension $L/k$. 
Keeping these assumptions, one gets the following corollary:
\begin{corollary}\label{Corollary_LS_4.3} \cite[Cor. 4.4]{LangeSchreieder2024}
    If $\bar X$ admits a $\Lambda$-decomposition of the  diagonal relative to $W_{\bar X}$, then the map 
    $$\Psi_{Y_L^\circ}^\Lambda:\bigoplus_{l\in I}\CH_1(Y_l^\circ\times_k L,\Lambda)\to \bigoplus_{\stackrel{i,j\in I}{i<j}}\CH_0(Y_{ij}^\circ\times_k L,\Lambda)$$
    is surjective for every field extension $L/k$. 
\end{corollary}

    \section{Essential $\CH_0$-dimension and relative decomposition of the diagonal}
 The following definition is due to Voisin.
 \begin{definition}\cite[Def. 1.2]{Voisin_Cubics2017}\label{def_essdim}
     Let $k$ be a field and $X$ be a $k$-variety. The group $CH_0(X)$ is \textit{universally supported} on a closed subscheme $Y\subset X$ if the pushforward $CH_0(Y_L)\rightarrow CH_0(X_L)$ is surjective for any field extension $L/k$.   The \textit{essential $\CH_0$-dimension}, or shorter \textit{essential dimension}, of $X$, denoted by $\essdim X$, is the minimal integer $n$ such that there exists a closed $n$-dimensional subscheme $Y\subset X$ such that $CH_0(X)$ is universally supported on $Y$.
 \end{definition}
 \begin{lemma}\label{Lemma_ess_dim}
     Let $k$ be a field of characteristic zero and $X$ a smooth proper $k$-variety. Then 
     \begin{equation}\label{essdim=relativetor}
        {\essdim X}=\mathrm{min}\{\dim Y\,|\,Y\subset X\,\text{ closed and}\,\Tor^\bb Z(X,Y)=1\}.
    \end{equation}
 \end{lemma}
 \begin{proof}
   The inequality LHS$\geq$RHS is clear: If $CH_0(X)$ is universally supported on $Y$, then taking $L=k(X)$, it follows that $X$ admits a decomposition of the diagonal with respect to $Y$, equivalently $\Tor^\bb Z(X,Y)=1$. Let $Y\subset X$ be a closed subscheme such that $\Tor(X,Y)= 1$. By \cite[Lemma 3.7]{LangeSchreieder2024}, and since $X$ is smooth, this implies that $CH_0((X-Y)_{L})$ is one-torsion for all field extensions $L/k$. In particular, considering the localization sequence 
      \begin{equation*}
       CH_0(Y_L)\rightarrow CH_0(X_L)\rightarrow CH_0((X-Y)_{L})\rightarrow 0
   \end{equation*}
   we get that the pushforward 
   \begin{equation*}
       CH_0(Y_L)\rightarrow CH_0(X_L)\rightarrow 0
   \end{equation*}
   is surjective. Hence, $CH_0(X)$ is universally supported on $Y$ and RHS$\geq$LHS.
 \end{proof}
 The essential dimension has been computed for complex varieties in the following cases.
 \begin{theorem}\label{known_ess_dim}
 Let $k=\bb C$.
     \begin{enumerate}
         \item \cite[Thm. $1.3$]{Voisin_Cubics2017}
         Let $X$ be a very general cubic hypersurface that does not admit a decomposition of the diagonal, then $\essdim X=\dim X.$
         \item \cite[Prop. 2.5]{CT2017cubiques} Let $X$ be a smooth projective variety of dimension at least two that satisfies $CH_0(X)\cong \bb Z$. If $X$ does not admit a decomposition of the diagonal, then
         $\essdim X\geq 2$.
         \item \cite[Thm. 0.9]{Mboro2019}
         Let $X$ be a very general Fano complete intersection threefold that does not admit a decomposition of the diagonal, then $\essdim X=3=\dim X.$
     \end{enumerate}
 \end{theorem}
 
Retract rational varieties admit a decomposition of the diagonal \cite[Lem. 7.5]{Schreiedersurvey}, hence their essential dimension is zero. Moreover, relating the essential dimension to the following definition by Chatzistamatiou--Levine we see that essential dimension is a stable birational invariant.
\begin{definition}\cite[Def. 1.1]{ChatzistamatiouLevine2018}\label{definition_tor_i}
    Let $k$ be a field and $X$ be $k$-scheme of finite type and of pure dimension $n$. Let $0 \leq i\leq n$ be an integer and consider the following group
    \begin{equation*}
        CH_n(X(i)\times X(n-1)):=\underset{W,D}{\varinjlim}CH_d((X-W)\times (X-D))
    \end{equation*}
    where $W,D$ vary over all closed subschemes of $X$ with $\dim_kW\leq i$ and $\dim_kD\leq n-1$. The \textit{$i^\mathrm{th}$torsion order of $X$}, denoted by $\Tor^i_k(X)\in \bb N_+\cup\{\infty\}$ is the order of the image of the class $\Delta_X\in CH_n(X\times X)$ under the morphism 
    \begin{equation*}
        CH_n(X\times X)\to  CH_n(X(i)\times X(n-1))
    \end{equation*}
    induced by the restrictions.
\end{definition}
\begin{lemma}\label{lemma_invariance_of_essdim}
    Let $k$ be a field of characteristic zero and $X$ a smooth proper $k$-variety. Then 
    \begin{equation*}
        \essdim X=\min \{i\geq 0\,|\,\Tor^{(i)}(X)=1\}.
    \end{equation*}
    In particular, the essential dimension is a stable birational invariant of proper $k$-varieties. 
\end{lemma}
\begin{proof}
By Lemma \ref{Lemma_ess_dim} we have that ${\essdim X}=\mathrm{min}\{\dim Y\,|\,Y\subset X\,\text{ closed and}\,\Tor^\bb Z(X,Y)=1\}.$
Then the first claim follows from the equality 
\begin{equation*}
    \Tor^{(i)}(X)=\min \{\Tor^\bb Z(X,W)\,|\,W\subset X\,\,\text{closed and}\,\\\dim W\leq i\}
\end{equation*}
which can be seen by the commutative diagram
    \[\begin{tikzcd}
	{CH_0(X_{K(X)})} & {\indlim_{D\subset X}CH_n(X\times (X-D)}) \\
	{\indlim_{W\subset X}CH_0((X-W)_{K(X)})} & {\indlim_{W,D\subset X}CH_n((X-W)\times (X-D)})
	\arrow["\cong", from=1-1, to=1-2]
	\arrow[from=1-1, to=2-1]
	\arrow[from=1-2, to=2-2]
	\arrow["\cong", from=2-1, to=2-2]
\end{tikzcd}\]
where the horizontal maps are isomorphisms and $W$ and $D$ are as in Definition \ref{definition_tor_i}.

The second claim follows from the fact that the $i$-th torsion order is a stable birational invariant by \cite[Prop. 2.7]{ChatzistamatiouLevine2018}, note that smoothness needs not be assumed as resolution of singularities holds.
\end{proof}
\begin{remark}
  By Lemma \ref{lemma_invariance_of_essdim} it follows that any ruled $n$-dimensional variety $X$ has $\essdim (X)<n$. By \cite{Kollar1996} it is known that for $d\geq 2\lceil\frac{n+3}{3}\rceil$ a very general hypersurface $X_d$ of degree $d$ in $\bb P^{n+1}$ is not ruled. Under the slightly lower bound $d\geq 2\lceil\frac{n+2}{3}\rceil$ the variety $X_d$ satisfies $2\mid \Tor^{(n-2)}(X_d)$ by \cite{ChatzistamatiouLevine2018}. In this case it is not known if $X_d$ is ruled or not.
\end{remark}

\section{Varieties with simple Chow groups}
For the following lemma, see also \cite[Lem. 3.11, Rem. 3.13]{LangeSkauli2023} and \cite[Lem. 5.8]{LangeSchreieder2024}.
\begin{lemma}\label{Lemma_simple_Chow}
    Let $k$ be an algebraically closed field and $X\subset \bb P^n_k=\Proj k[x_0,\dots,x_{n}]$ a hypersurface of degree $d$ given by the following equation $$
    X=\{P(x_{0},\dots,x_{n-1})x_n+H(x_0,\dots,x_{n-1})=0\}$$
    with $P\neq 0$.
    Then $X-\{P=0\}\cong \bb P_k^{n-1}-\{P=0\}$. In particular, if $\{P=0\}\subset \bb P_k^{n-1}$ contains a line, then
    $$CH_1(X-\{P=0\})=0$$
   and the natural pushforward map
    $$\CH_1(\{P=0\}\cap X)\to \CH_1(X)$$
    is surjective.
        \end{lemma}
        \begin{proof}
    Consider the projection away from $(0:...:0:1)$
    \begin{align*}
    \pi:X&\dashrightarrow \bb P_k^{n-1}=\{x_n=0\}\\
    (x_0:....:x_n)&\mapsto (x_0:....:x_{n-1})
    \end{align*}
   Restricting $\pi$ to $\pi|_{X-\{P=0\}}:X-\{P=0\}\rightarrow \bb P_k^{n-1}-\{P=0\}$ gives an isomorphism with inverse \begin{align*}
    \phi: \bb P_k^{n-1}-\{P=0\}&\rightarrow X-\{P=0\}.\\
    (x_0:....:x_{n-1})&\mapsto (x_0:...:x_{n-1}:-\frac{H(x_{0},\dots,x_{n-1})}{P(x_{0},\dots,x_{n-1})})
    \end{align*} 
    The second claim of the lemma then follows from the exact sequence
\[\begin{tikzcd}
	{CH_1(\{P=0\})} & {CH_1(\bb P_k^{n-1})} & {CH_1(\bb P_k^{n-1}-\{P=0\})} & 0
	\arrow[from=1-1, to=1-2]
	\arrow[from=1-2, to=1-3]
	\arrow[from=1-3, to=1-4]
\end{tikzcd}\]
and the fact that $CH_1(\bb P_k^{n-1})\cong [\gamma]\bb Z$, where $\gamma$ is a line defined over $k$ and the last claim then follows from the exact sequence
\[\begin{tikzcd}
	{CH_1(\{P=0\}\cap X)} & {CH_1(X)} & {CH_1(X-\{P=0\})} & 0.
	\arrow[from=1-1, to=1-2]
	\arrow[from=1-2, to=1-3]
	\arrow[from=1-3, to=1-4]
\end{tikzcd}\]
    \end{proof}
\section{Main theorem}
We recall the motivation for our main result.
\begin{theorem}\label{cubic_to_quartic}\cite[Thm. 4.4, Cor. 7.9]{NicaiseOttem2022}
    Let $k=\bb C$ and $n\geq 0$, $d\geq d'$ be positive integers. Assume that a very general degree $d'$ hypersurface in $\bb P^{n+1}_k$ is stably irrational. Then a very general member of the following families 
\begin{enumerate}
 \item hypersurfaces of degree $d+1$ in $\bb P^{n+2}_k$,
\item complete intersections of type $(2,d)$ in $\bb P^{n+3}_k$,
    \end{enumerate}
    is not stably rational.
\end{theorem}
We upgrade Theorem \ref{cubic_to_quartic} to a statement about the decomposition of the diagonal when $d'=3$. The restriction to $d'=3$ is due to the fact that we only know the essential dimension of cubic hypersurfaces.
\Maintheorem*
\begin{definition}\label{def_degenration_family_for_(d)}
    Let $k$ be an algebraically closed field of characteristic zero.
\begin{enumerate}
    \item Let $3\leq n$, $4\leq d\leq n+2$ and
    \begin{equation*}
       f,c,g,M,L\in k[x_0,...,x_{n+1}]
    \end{equation*}
    be pairwise different homogeneous polynomials that define very general hypersurfaces of the following degrees
    \begin{equation*}
        \deg(f)=2,\,\,\deg(c)=3,\,\,\deg(g)=d-4,\,\,\deg(M)=\deg(L)=d-2.
    \end{equation*}
    Let
    \begin{align*}
        H:=gc+Mx_{n+2}+Lx_{n+3} \in k[x_0,\dots,x_{n+3}].
    \end{align*}
        \item 
    Let $R=k[t]_{(t)}$, \begin{align*}
        F:=x_{n+2}x_{n+3}+tf \in R[x_0,\dots,x_{n+3}]
    \end{align*}
    and consider the $R$-scheme 
    \begin{equation*}
     \cal X:=\{H=F=0\}\subset \bb P^{n+3}_R.
    \end{equation*}
    Let $K=\mathrm{Quot}(R)$. We denote the generic fiber by $X=\cal X\times K$.
    The special fiber $Y=\cal X\times _R k$ has the following two components:
    \begin{align*}
        Y_0:=\{gc+Mx_{n+2}=0\}\subset \bb P_k^{n+2},\\
        Y_1:=\{gc+Lx_{n+3}=0\}\subset \bb P_k^{n+2}.
    \end{align*}
    The intersection $Z:=Y_0\cap Y_1$ is the reducible degree $d-1$ hypersurface 
\begin{equation*}
    Z:=\{cg=0\}\subset \bb P^{n+1}_k.
\end{equation*}
\end{enumerate}
\end{definition}
\begin{lemma}\label{(2,d)birational(d)}
  The geometric generic fiber $\bar X:= X_{\bar K}$, given by 
   \begin{equation*}
      \{H=F=0\}\subset \bb P^{n+3}_{\bar{K}},
   \end{equation*}
   is integral and birational to the degree $d$ hypersurface
    \begin{equation*}
   Q:= \{q:=x_{n+2}cg+x_{n+2}^2M-Ltf=0\}\subset \bb P^{n+2}_{\bar K}.
\end{equation*}
\end{lemma}
\begin{proof}
To see that the two $\bar K$-schemes are birational to each other restrict to the principal open $\mathrm{D}_+( x_{n+2}) \subset \mathbb{P}^{n+3}_{\bar K}$, resp. $\mathrm{D}_+( x_{n+2}) \subset \mathbb{P}^{n+2}_{\bar K}$.
We first show that $Q$ is integral. For this it suffices to prove that the homogeneous polynomial 
\begin{equation*}
    q(x_0,...,x_{n+1},x_{n+2})\in \bar K[x_0,...,x_{n+1},x_{n+2}]=A[x_{n+2}],
\end{equation*}
where $A=\bar K[x_0,...,x_{n+1}]$, is irreducible. As $A$ is a UFD and $q$ is primitive (by the generality of the coefficients), its irreducibility follows from Gauss's lemma and the fact that the discriminant 
\begin{equation*}
    \Delta_q=(cg)^2+4MLtf\in \mathrm{Quot}(A)
\end{equation*}
is not a square (by the generality of the coefficients and the assumption $\mathrm{char}(k)\neq 2$). To see this assume $s^2=\Delta_q$. By \cite[Thm 3.1.1]{valuedfields} the $t$-adic valuation extends to $\bar K$. Consider the valuation ring $\cal O_t\subset \bar K$ of the extension. The residue field of $\cal O_t$ is the algebraically closed base field $k$. We claim that under the previous assumption $cg$ divides $LMf$ as $k$-polynomials, which is impossible by the very generality of the polynomials and degree reasons. From the equality $s^2=(cg)^2+tLMf$ it follows that $s=cg+P$ for some element $P\in \cal O_t$ in the maximal ideal of $\cal O_t$. In fact the $t$-adic valuation of $P$ is $1$. Hence we have $\frac{P}{t}(2cg+P)=LMf$ in $\cal O_t$. Taking the residue in the residue field $k$ then implies the claim. To see that $\bar X$ is integral, it suffices to see that the closed subscheme $ \{x_{n+2}=0\}\subset \bb P^{n+3}_{\bar K}$ does not contain an irreducible component of $\bar X$. This is so because $\bar X$ is equidimensional of dimension $n+1$ and as one can see by inspecting $F$ and $H$ the scheme $\{F=H=x_{n+2}=0\}$ has dimension $n$. 
\end{proof}

\begin{definition}\label{Definition_W_X2}
    Let $$W_\cal X:=  \{MLfg=0\}\subset\cal X.$$
    Denote by $W_Y$ the special fiber and by $W_X$ (resp. $W_{\bar{X}}$) the generic (resp. geometric generic) fiber of $W_\cal X$. The complements are denoted as follows \begin{equation*}
        Y^\circ:=Y\setminus W_Y,\,\,Y_i^\circ:=Y_i\setminus W_Y,\,\,Z^\circ:=Z\setminus W_Y,\,\,X^\circ:=X\setminus W_X.
    \end{equation*}
\end{definition}
From now on let $\Lambda$ be a ring of positive characteristic such that the exponential characteristic of $k$ is invertible in $\Lambda$.
\begin{lemma}\label{Lemma_CH=0_2}
    $$\CH_1(Y_{0}^\circ\times_k L,\Lambda)=0=\CH_1(Y_{1}^\circ\times_k L,\Lambda)$$ 
    for any field extension $L/k$.
\end{lemma}
\begin{proof}
    This follows immediately from Lemma \ref{Lemma_simple_Chow} since $\{L=0\}\subset \bb P_k^{n+1}$ and $\{M=0\}\subset \bb P_k^{n+1}$ both contain a line by \cite[Lem. $2.9$]{BeheshtiRoya2013}.
\end{proof}
\begin{lemma}\label{boundaries of family}
    The $k$-varieties $Y_0^\circ,\,Y_1^\circ,\,$ and $Z^\circ$ are smooth $k$-varieties and $\dim(Y_0^\circ)=\dim(Y_1^\circ)=n+1$ and $\dim(Z^\circ)=n$.
\end{lemma}
\begin{proof}
    Taking the derivative with respect to $x_{n+2}$ (resp. $x_{n+3}$) of the equation defining $Y_0$ (resp. $Y_1$), we get
    $$\partial_{x_{n+2}}(gc+Mx_{n+2})=M$$
    (resp. $\partial_{x_{n+3}}(gc+Lx_{n+3})=L$) whose zero locus is contained in $W_{Y}$. To see that $Z^\circ$ is smooth, note that 
    \begin{equation*}
        Z^{\circ}=\{gc=0\}-W_Y\subset \bb P^{n+1}_k
    \end{equation*}
    is an open subscheme of the very general, and hence smooth, cubic hypersurface $C:=\{c=0\}\subset \bb P_k^{n+1}$ and nonempty by the generality of the components of $W_Y$. 
\end{proof}

\begin{lemma}\label{singular_locus_calX_2}
    The singular locus of $\cal X$ is contained in $\{Lf=0\}\subset \bb P^{n+3}_R$.
\end{lemma}
\begin{proof}
 The singular locus of $\cal X$ is given by the vanishing of the equations defining $\cal X$ and the locus where all minors of the Jacobian vanish. 
 Set $x=(x_0,...,x_{n+1})$. The Jacobian of $\cal X$ is given by the matrix
\begin{align*}
Jac\,\cal X=
\begin{pmatrix} \nabla F \\ \nabla H \end{pmatrix}&=
\begin{pmatrix}
\displaystyle \frac{\partial F}{\partial x} &
\displaystyle \frac{\partial F}{\partial x_{n+2}} &
\displaystyle \frac{\partial F}{\partial x_{n+3}} &
\displaystyle \frac{\partial F}{\partial t}
\\[10pt]
\displaystyle \frac{\partial H}{\partial x} &
\displaystyle \frac{\partial H}{\partial x_{n+2}} &
\displaystyle \frac{\partial H}{\partial x_{n+3}} &
\displaystyle \frac{\partial H}{\partial t}
\end{pmatrix}
\\
\\
&=\begin{pmatrix}
\,t\,\nabla_{x} f &
 x_{n+3} &
 x_{n+2} &
\,f
\\[6pt]
\nabla_{x} cg+x_{n+2}\nabla_x M+ x_{n+3}\nabla_x L &
M &
L  &
0
\end{pmatrix}
\end{align*}
and the claim follows if we consider the minor on right given by $x_{n+2}\cdot 0-Lf=-Lf$.
\end{proof}
\begin{lemma}\label{Lemma_semi_stability_2}
    $\cal X^\circ=\cal X\setminus W_{\cal X}$ is strictly semistable.
\end{lemma}
\begin{proof} (Cf. \cite[Prop. 5.9, Step 1]{LangeSchreieder2024})
By Lemma \ref{boundaries of family}, $Y^\circ$ is an snc scheme. Since by Lemma \ref{singular_locus_calX_2} $\cal X^{\mathrm{sing}}\subset W_{\cal X}$ the total space $\cal X^\circ$ is regular. 
Since on a regular scheme Weil divisors are Cartier, $Y_0^\circ,Y_1^\circ$ are Cartier divisors on $\cal X^\circ$. Therefore, by definition (see \cite[Def. 1.1]{Hartl2001}), $\cal X^\circ=\cal X\setminus W_{\cal X}$ is strictly semistable.  
\end{proof}

\begin{proof}[Proof of Thm. \ref{Theorem_2_4_complete_int_Intro}(i)]
By \cite[Lem. 2.2]{Totaro2016} we can assume that $4\leq d\leq n+2$. Let $\cal X$ be as in Definition \ref{def_degenration_family_for_(d)}. Denote the smooth very general cubic by
\begin{equation*}
    C:=\{c=0\}\subset \bb P_k^{n+1}.
\end{equation*}
By our assumption on the diagonal and the essential dimension of $C$ (Theorem \ref{known_ess_dim}), Lemma \ref{lemma_invariance_of_essdim} implies that 
\begin{equation*}
    1<\Tor^{\bb Z}(C, C\cap W_Y)=\Tor^{\bb Z}( Z^{\circ},\emptyset).
\end{equation*} 
Note that the equality on the right follows from Rem. \ref{Remark_LS_3.4} and the fact that $Z^{\circ}=Z-C\cap W_Y=\{gc=0\}-\{g=0\}$.  As $C$ is smooth 
we can apply \cite[Lemma 3.7]{LangeSchreieder2024} and see that $CH_0(({Z^{\circ}})_{k(Z)})$ is $\Tor^{\bb Z}( Z^{\circ},\emptyset)$-torsion. Setting $\Lambda:=\bb Z/\Tor^{\bb Z}( Z^{\circ},\emptyset)$ and considering the exact sequence  $$\CH_0(({Z^{\circ}})_{k(Z)})\xto{\cdot \Tor^{\bb Z}( Z^{\circ},\emptyset)}\CH_0(({Z^{\circ}})_{k(Z)})\to \CH_0(({Z^{\circ}})_{k(Z)},\Lambda)\to 0$$
this implies that
    \begin{equation*}
        CH_0(({Z^{\circ}})_{k(Z)},\Lambda)=CH_0(({Z^{\circ}})_{k(Z)})
    \end{equation*} and hence that $$1<\Tor^\bb Z(Z,W_Z)=\Tor^{\Lambda}(Z,W_Z).$$ 
Let $$m:=\Tor^\Lambda(\bar X,W_{\bar X}).$$
 Since by Lemma \ref{(2,d)birational(d)} the generic fiber $X_K$ is geometrically integral and by Lemma \ref{Lemma_semi_stability_2} the $R$-scheme $\cal X^\circ$ is strictly semistable, and its special fiber has no triple intersections, we can apply
    Theorem \ref{Theorem_LS_4.1}. Theorem \ref{Theorem_LS_4.1} 
    then says that the cokernel of the map 
$$\Psi_{Y_L^\circ}^\Lambda:\CH_1(Y_0^\circ\times_k L,\Lambda)\oplus \CH_1(Y_1^\circ\times_k L,\Lambda)\to \CH_0(Z^\circ\times_k L,\Lambda)$$
    is $m$-torsion for every field extension $L/k$.
    In particular, $$m\delta_{Z}|_{Z^\circ}=\gamma_0|_{Z^\circ}-\gamma_1|_{Z^\circ}$$
    for some $\gamma_i\in \CH_1(Y_i^\circ\times_k L,\Lambda)$ and $L=k(Z^\circ)$. Since by Lemma \ref{Lemma_CH=0_2}, the source of $\Psi_{Y_L^\circ}^\Lambda$ is zero, we get that $m\delta_{Z}|_{Z^\circ}=0$ and by the above this implies that $m>1$.\footnote{Indeed, this implies that $2=\Tor(C)|m$; see Remark \ref{remark_2_divisibility}.} Let 
\begin{equation*}
    Q:=\{q=0\}\subset \bb P^{n+2}_k
\end{equation*}
be the degree $d$ hypersurface defined in Lemma \ref{(2,d)birational(d)}.
Restricting the birational map $\phi:\bar X\dashrightarrow Q$ to $\bar X-W_{\bar X}$, we get that 
$$\Tor^\Lambda(\bar X,W_{\bar X})=\Tor^\Lambda(Q,Q-\phi (\bar X-W_{\bar X})).$$ 
The equality follows again from Rem. \ref{Remark_LS_3.4} since $\bar X-W_{\bar X}\cong Q-(Q-\phi (\bar X-W_{\bar X}))$. For this note that $\phi$ restricted to $\bar X-W_{\bar X}$ is a morphism because its indeterminacy locus $\mathrm{ind}(\phi)=\{x_{n+2}=0\}\cap \bar X$
    is contained in $W_{\bar X}$. Indeed, if $x_{n+2}=0$, then $0=\frac{x_{n+2}x_{n+3}}{t}=f$ (or, equivalently, if $f\neq 0$ on $\bar X$ fiber, then $F=0$ implies $x_{n+2}\neq 0$). From now on we denote the closed subset $Q-\phi (\bar X-W_{\bar X})$ by $W_Q$.
    
    By \cite[Lem. 8]{Schreieder2019quadricbundles} there exists a very general (and therefore smooth) degree $d$ hypersurface $Q'$ which specializes to $Q$ in the sense of \cite[\S 2.2]{Schreieder2019quadricbundles}. In other words, there exists a proper flat surjective morphism $$f:\cal Y\to \Spec B$$ where $B$ is an equicharacteristic DVR with residue field $\bar K$ such that $Q\cong \cal Y_{\bar K}$ and there exists an injection of fields $\Frac(B)\hookrightarrow \bar K$ (up to field isomorphism) with
    $$\cal{Y}_{\Frac B}\times\bar{K}\cong  Q'.$$
    Let $B^h$ be the henselisation of $B$ and $f':\cal Y'\to \Spec B^h$ the base change of $f$ along $\Spec B^h\to \Spec B$. Then $Q$ is still isomorphic to $\cal Y_{\bar K}'$ and we denote the generic fiber of $\cal Y'$ by $Q^h$. Let $P\in Q$ be a regular closed point contained in $W_Q$ and $T\subset \cal Y'$ a lift of $P$ over $\Spec B^h$. Setting $W_{\cal Y'}:=T\cup W_Q$, \cite[Lem. 3.8]{LangeSchreieder2024} and the fact that $T$ specializes to $Q-\phi (\bar X-W_{\bar X})$ imply that
    $$\Tor^\Lambda(Q,W_Q)\;|\;\Tor^\Lambda(Q^h_{ \overline{\Frac(B^h)}},W_{\cal Y'}\times \overline{\Frac(B^h)})$$
    $$|\;\Tor^\Lambda(Q^h_{ \overline{\Frac(B^h)}},T\times \overline{\Frac(B^h)}),\footnote{ Alternatively, instead of passing to the henselisation, we could choose $\dim Q$ general hyperplane sections through a closed point of $Q-\phi (\bar X-W_{\bar X})$ and lift these to a closed subset $T\subset \cal Y$ of relative dimension zero over $B$. In order for $T$ to specialize to $Q-\phi (\bar X-W_{\bar X})$, we would then have to add a set of closed points $S_K$ to $W_{\bar X}$ by modifying $W_{\cal X}$ in Definition \ref{Definition_W_X2} accordingly (cf. \cite[Proof of Thm. 7.1]{LangeSchreieder2024}).}$$
    where the second divisibility follows from Lem. \ref{Lemma_LS_3.6}(b).
     Finally, since $\deg:\CH_0(Q^h_{ \overline{\Frac(B^h)}})\to \bb Z$ is an isomorphism ($Q^h_{ \overline{\Frac(B^h)})}$ is Fano over an algebraically closed field),  by Lemma \ref{Lemma_LS_3.6}(d) $\Tor(Q^h_{ \overline{\Frac(B^h)}})=\Tor^\bb Z(Q^h_{ \overline{\Frac(B^h)}},P)$, where $P$ is the closed point $T_{ \overline{\Frac(B^h)}}$. In other words, $Q^h_{ \overline{\Frac(B^h)}}$ does not admit a decomposition of the diagonal. Considering $Q^h_{ \overline{\Frac(B^h)}}$ as a scheme over $\bb C$ via an isomorphism $\bb C\cong\bar K\cong \overline{\Frac(B^h)}$, we have found a smooth complex degree $d$ hypersurface which does not admit a decomposition of the diagonal. This implies that the same holds for a very general degree $d$ hypersurface. Retract rational and $\bb A^1$-connected varieties admit a decomposition of the diagonal by \cite[Lem. 2.4]{SchreiederAMS}.
\end{proof}

\begin{proof}[Proof of Thm. \ref{Theorem_2_4_complete_int_Intro}(ii)]
     By \cite[Lem. 2.2]{Totaro2016} we can assume that $4\leq d\leq n+2$. Let $\cal X$ be as in Definition \ref{def_degenration_family_for_(d)}. In the proof of Theorem \ref{Theorem_2_4_complete_int_Intro}(i) we construct a $(2,d-1)$ complete intersection $\bar X$ which does not admit a decomposition of the diagonal. As complete intersections are parametrised by an open subscheme of a projective space, we can conclude the claim following the last part of the proof of Theorem \ref{Theorem_2_4_complete_int_Intro}(i).
     \end{proof}
\begin{proof}[Proof of Corollary \ref{Corollary_main_Intro}]
    The statement follows from Theorem \ref{Theorem_2_4_complete_int_Intro} and \cite[Cor. 1.4]{EFS2025}. 
\end{proof}
\begin{proof}[Proof of Corollary \ref{cor:unir_retract_irr}]
First recall that the parameter spaces of the families (i) and(ii) are irreducible. The claims on the unirationality of a general member follow from \cite[Thm. 3.15]{Massarenti2023} for family $(i)$ and from \cite[Thm. 4.10(iii)]{Massarenti2023} for family $(ii)$.  
In order to show that a very general member does not admit a decomposition of the diagonal it suffices in both cases to show that there is such a member $X$ with no decomposition of the diagonal relative to a point since the unirationality implies rational connectedness and therefore that $\CH_0(X)\cong \bb Z$. By the same argument as at the end of the proof of Theorem \ref{Theorem_2_4_complete_int_Intro} it suffices for this to find a special member with no decomposition of the diagonal relative to a point. 

For the case of family $(i)$ we note that the hypersurface $Q$ described in Lemma \ref{(2,d)birational(d)} for $n=3$ and $d=4$ has a double point at $(0:\dots:0:1)$, hence the proof of Theorem \ref{Theorem_2_4_complete_int_Intro} shows that $Q$ does not admit a decomposition of the diagonal relative to a point. 

For the family $(ii)$, we need to adapt the construction of Definition \ref{def_degenration_family_for_(d)} slightly. Let $c,L,M\in k[x_0,\dots,x_4]$ be very general coprime polynomials of degree $3$, $g=x_0$ and $f=x_0^2$. In other words, the equation is slightly less general since $f$ and $g$ are now specifically chosen, but still $cg$ does not divide $LMf$. Then exactly as in Lem. \ref{(2,d)birational(d)} one constructs the quintic fourfold $X=\{q:=x_5x_0c+Mx_5^2-tLx_0^2=0\} \subset \bb P_{\bar{K}}^5$. $X$ contains the linear subspace $\{x_5=0=x_0\}=\bb P^3_{\bar{K}}$ with multiplicity two since $q\in (x_0,x_5)^2=(x_0^2,x_0x_5,x_5^2)$. 
The argument of Theorem \ref{Theorem_2_4_complete_int_Intro} now carries over verbatim to show that $X$ does not admit a decomposition of the diagonal relative to a point.
\end{proof}
\bibliographystyle{acm}
\bibliography{Bibliografie1}

\noindent
\parbox{0.5\linewidth}{
\noindent
Elia Fiammengo \\ 
Universität Heidelberg\\
Mathematisches Institut \\
Im Neuenheimer Feld 205 \\
69120 Heidelberg \\
Germany\\
{\tt 	elia.fiammengo@stud.uni-heidelberg.de}
}
\\
\newline

\noindent
\parbox{0.5\linewidth}{
\noindent
Morten L\"uders \\ 
Universität Heidelberg\\
Mathematisches Institut \\
Im Neuenheimer Feld 205 \\
69120 Heidelberg \\
Germany\\
{\tt mlueders@mathi.uni-heidelberg.de}
}

\end{document}